\documentclass[11pt]{article}
\usepackage{tocbibind} % adds References to the pdf bookmarks, in amsart this seems to be standard, but not in article
\usepackage{amsmath,amsthm,amssymb}
\usepackage{fancyhdr}
\usepackage[british]{babel}
\usepackage{geometry,mathtools}
\usepackage{enumitem}
\usepackage{algpseudocode}
\usepackage{dsfont}
\usepackage{centernot}
\usepackage{xstring}
\usepackage{colortbl}

\usepackage[section]{algorithm}
\usepackage{graphicx}
\usepackage[font={footnotesize}]{caption}
\usepackage[usenames,dvipsnames,table]{xcolor}
\usepackage{pstricks,tikz}
\usetikzlibrary{patterns,calc,fadings,decorations.pathreplacing}
\usepackage[h]{esvect}
\usepackage[
bookmarksopen=true,
bookmarksopenlevel=1,
colorlinks=true,
linkcolor=darkblue,
linktoc=page,
citecolor=darkblue,
urlcolor=darkblue
]{hyperref}
\usepackage{bbm}
\usepackage{cleveref}
\numberwithin{equation}{section}

\definecolor{darkblue}{rgb}{0,0,0.5}

\newdimen\margin
\def\textno#1&#2\par{
   \margin=\hsize
   \advance\margin by -4\parindent
          \setbox1=\hbox{\sl#1}
   \ifdim\wd1 < \margin
      $$\box1\eqno#2$$
   \else
      \bigbreak
      \hbox to \hsize{\indent$\vcenter{\advance\hsize by -3\parindent
      \it\noindent#1}\hfil#2$}
      \bigbreak
   \fi}

\newtheorem{theorem}[algorithm]{Theorem}
\newtheorem{prop}[algorithm]{Proposition}
\newtheorem{lemma}[algorithm]{Lemma}
\newtheorem{cor}[algorithm]{Corollary}

\theoremstyle{definition}

\newtheorem{defin}[algorithm]{Definition}
\newtheorem{remark}[algorithm]{Remark}

\def\claimproof{\removelastskip\penalty55\medskip\noindent{\em Proof of claim: }}

\newcounter{stepenv}
\newenvironment{stepenv}[1][]{\refstepcounter{stepenv}}{}

\newcounter{step}[stepenv]

\newcounter{substep}[step]
\renewcommand{\thesubstep}{\thestep.\arabic{substep}}

\newcounter{claim}
\newenvironment{claim}[1][]{\refstepcounter{claim}\par\medskip\noindent%
        \textit{Claim~\theclaim. #1} \itshape\rmfamily}{\medskip}

\newcommand{\cF}{\mathcal{F}}

\def\eps{{\varepsilon}}

\newcommand{\prob}[1]{\mathrm{\mathbb{P}}\left[#1\right]}

\newcommand{\expn}[1]{\mathrm{\mathbb{E}}\left[#1\right]}

\def\sm{\setminus}
\newcommand{\Set}[1]{\{#1\}}

\def\In{\subseteq}

\def\COMMENT#1{}
\def\TASK#1{}
\let\TASK=\footnote             % COMMENT OUT for clean output
\let\COMMENT=\footnote          % COMMENT OUT for clean output

\begin{document}

\title{Optimal Hamilton covers and linear arboricity for random graphs}

\author{Nemanja Dragani\'c \thanks{
Department of Mathematics, ETH, Z\"urich, Switzerland. Research supported in part by SNSF grant 200021\_196965.
\emph{E-mails}: \textbf{\{nemanja.draganic,david.munhacanascorreia,benjamin.sudakov\}@math.ethz.ch}.
}
\and Stefan Glock \thanks{Fakultät für Informatik und Mathematik, Universität Passau, Germany.
\emph{Email}: \textbf{stefan.glock@uni-passau.de}}
\and David Munh\'a Correia \footnotemark[1]
\and Benny Sudakov \footnotemark[1]
}

\date{}

\maketitle
\begin{abstract} 
In his seminal 1976 paper, P\'osa showed that for all $p\geq C\log n/n$, the binomial random graph $G(n,p)$ is with high probability Hamiltonian. This leads to the following natural questions, which have been extensively studied: How well is it typically possible to cover all edges of $G(n,p)$ with Hamilton cycles? How many cycles are necessary? In this paper we show that for $ p\geq  C\log n/n$, we can cover $G\sim G(n,p)$ with precisely $\lceil\Delta(G)/2\rceil$ Hamilton cycles. Our result is clearly best possible both in terms of the number of required cycles, and the asymptotics of the edge probability $p$, since it starts working at the weak threshold needed for Hamiltonicity. This resolves a problem of Glebov, Krivelevich and Szabó, and improves upon previous work of Hefetz, K\"uhn, Lapinskas and Osthus, and of Ferber, Kronenberg and Long, essentially closing a long line of research on Hamiltonian packing and covering problems in random graphs.
\end{abstract}

\section{Introduction}
Given two (hyper)graphs $H$ and $G$, can we partition the edges of $G$ into copies of $H$? This is the general framework of graph decomposition problems, which have been extensively studied for various instances of graphs $H$ and $G$. One of the most well-known results of this type was shown by Kirkman in 1847, who proved that the complete graph $K_n$ can be decomposed into copies of $K_3$ if and only if $n \equiv 1, 3 \mod 6$. The extension of this condition to general graphs $F$ instead of $K_3$ remained an unsolved problem for a century until Wilson resolved it in 1975. He showed that for every graph~$F$, if $n$ is large enough and certain (necessary) divisibility conditions are satisfied, then $K_n$ has an $F$-decomposition. In the case of hypergraphs, solving a problem which goes back to the 19th century, Keevash \cite{keevash2014existence} showed that large enough complete hypergraphs can be decomposed into constant sized cliques, if divisibility conditions are satisfied, and the corresponding problem for general hypergraphs $F$ was settled in~\cite{Fdesigns}.

Since it is not always possible to decompose $G$ into copies of $H$, a natural question is how many edge-disjoint copies of $H$ can we \emph{pack} into $G$. To this end, we define an \emph{$H$-packing} as a collection of edge-disjoint copies of $H$ in $G$. Another avenue of study is that of covering problems, where one asks for the minimal number of copies of $H$ needed to \emph{cover} $G$, or in other words, what is the size of a minimal \emph{$H$-cover} in~$G$. Packing and covering problems are closely related, and in some instances are even trivially equivalent. For example, the famous Erd\H{o}s--Hanani problem \cite{erdos1963limit} on packing and covering complete $s$-uniform hypergraphs with $k$-cliques, solved by R\"odl \cite{rodl1985packing}, has the same asymptotic answer for both its covering and packing version. %of the problem. It was solved by R\"odl \cite{rodl1985packing}, before Keevash \cite{keevash2014existence} showed the stronger decomposition result.

In this paper, we are concerned with packings and coverings by Hamilton cycles. One of the oldest results of this flavour is Walecki's theorem from around 1890, stating that $K_{n}$ can be decomposed into Hamilton cycles if $n$ is odd, and into Hamilton cycles plus one perfect matching if $n$ is even. A far-reaching generalization of this result, when $K_n$ is replaced by a regular graph with degree at least $n/2$, was proven by Csaba, K\"uhn, Lo, Osthus and Treglown~\cite{CKLOT:16}, thereby confirming a conjecture of Nash-Williams \cite{nash1970hamiltonian} from 1970. 
Clearly, the bound on the degree is best possible since otherwise the graph might not even contain a Hamilton cycle.
%Kelly conjecture

Another natural graph class to study questions about Hamilton cycles are random graphs. Here, we consider the binomial random graph $G(n,p)$ which has $n$ vertices and edges are chosen to appear independently with probability $p$. We say that an event (or more precisely, a sequence of events) holds \emph{with high probability (whp)} if the probability of the event tends to $1$ as $n\to \infty$. The investigation of packings of Hamilton cycles in random graphs was initiated by Bollob\'as and Frieze \cite{bollobas1983matchings}. In the last decades, significant attention has been directed towards the problem of identifying optimal packings and coverings with Hamilton cycles within random graphs. This attention has resulted in the solution of the packing problem through a series of papers by multiple authors. In particular, it is clear that a packing of Hamilton cycles is of size at most $\lfloor\delta(G)/2\rfloor$; hence, the question becomes interesting for random graphs starting with $p\sim \frac{\log n}{n}$. The results obtained in \cite{bollobas1983matchings, knox2015edge, krivelevich2012optimal, kuhn2014hamilton} cover the whole range of $p$, thus confirming a conjecture of Frieze and Krivelevich~\cite{FK:08}.

\begin{theorem}[\cite{bollobas1983matchings, knox2015edge, krivelevich2012optimal, kuhn2014hamilton}]\label{thm:packing}
For any $p$, in $G\sim G(n,p)$ whp there exists a collection of $\lfloor\delta(G)/2\rfloor$ edge-disjoint  Hamilton cycles.
\end{theorem}

\noindent In contrast to the Erd\H{o}s--Hanani problem where the target graph for packing $H$ remains small compared to~$G$, the situation changes when we permit $H$ to grow in size with $G$. This leads to a departure from the equivalence between the (asymptotic) packing and covering problems. Indeed, Kuzjurin~\cite{kuzjurin:95} showed that asymptotically optimal packings exist only when the clique size is of order less than $\sqrt{n}$, while asymptotically optimal coverings exist up to clique size $o(n)$.

Given the success of the packing problem, it was natural to consider the ``dual'' question of covering $G(n,p)$ with a small number of Hamilton cycles. 
Somewhat surprisingly, the covering question has withstood the test of time to a greater extent than its corresponding packing version.
The covering prolem was initiated by Glebov, Krivelevich and Szab\'o~\cite{GKS:14}, who showed that whp $(1+o(1))np/2$ Hamilton cycles suffice to cover $G\sim G(n,p)$, whenever $p\geq n^{-1+\varepsilon}$ for any constant $\varepsilon>0$. They conjectured that the optimal number of $\lceil\Delta(G)/2\rceil$ Hamilton cycles can be attained and that further, this could be done whenever $p= \omega \left(\frac{\log n}{n} \right)$. Further, this question is the second problem listed in the well-known bibliography on Hamilton cycles in random graphs by Frieze~\cite{frieze2019hamilton}.

The result of Glebov,
Krivelevich and Szab\'o was subsequently improved by Hefetz, K\"uhn, Lapinskas and Osthus~\cite{HKLO:14}, who showed that for $\frac{\log^{117}n}{n}\leq p\leq 1-n^{-1/8}$ one can whp cover $G\sim G(n,p)$ with $\lceil\Delta(G)/2\rceil$ Hamilton cycles. Later, Ferber, Kronenberg and Long~\cite{ferber2017packing} were able to improve the range of $p$ for which the approximate covering result holds, showing that when $p\gg \frac{\log^2n}{n}$ then whp $G\sim G(n,p)$ can be covered with $(1+o(1))np/2$ Hamilton cycles. %The final range to be solved in the packing question was $\frac{C\log n}{n} \leq p \leq \log^Cn$. In order to achieve closure in the covering question as well, the same range proved to be the last one to be resolved.
We resolve this problem by proving that the conjecture of Glebov, Krivelevich and Szab\'o is true in the most interesting range, starting at the (weak) threshold for Hamiltonicity.

%One of the most famous examples is the Erdos-Hanani problem

% Hamiltonicity is one of the central notions in combinatorics, and has been extensively studied by many researchers. Since the problem of determining whether a graph is Hamiltonian is NP-complete, a important theme in Combinatorics is to derive sufficient conditions for this property. Once such a property is derived, the next question usually is the "robustness" of the Hamiltonicity of the graph.

% This paragraph about resilience, pancyclicity, counting, packing.

% Another well studied problem is that of covering the edges of a random graph with Hamilton cycles. It was raised by Glebov, Krivelevich and Szabó, who showed that for $\leq p\leq $ the binomial random graph $G(n,p)$ can be covered by at most $(1+o(1))np$ Hamilton cycles. They raised..

%Given a graph $G$, a Hamilton cover of $G$ is a set of Hamilton cycles in $G$ which together cover all edges of~$G$. It is optimal if it has size $\lceil \frac{\Delta(G)}{2}\rceil$.
\begin{theorem}\label{thm:main}
Let $\frac{C\log n}{n} \leq p(n) \leq n^{-2/3}$ for large enough $C$. Then whp $G(n,p)$ can be covered by $\lceil\Delta(G)/2\rceil$ Hamilton cycles.
\end{theorem}

%One can use our methods to cover almost the whole range for $p$,\textcolor{red}{sure about this?} but since the remainder of the range is already covered in \cite{HKLO:14}, for clarity of presentation we chose to present the proof in this shorter range.

\noindent Our result has an immediate corollary concerning the linear arboricity conjecture, which states that every graph $G$ can be decomposed into $\lceil \frac{\Delta(G)+1}{2}\rceil$ linear forests (see~\cite{linarbconj}).
%Recall that a \emph{linear forest} is collection of vertex-disjoint paths. The conjecture is still wide open.
At the moment the best result towards this is by Lang and Postle~\cite{LangPostle} who showed that $G$ can be decomposed into at most $\frac{\Delta}{2} + \Tilde{O} \left(\sqrt{\Delta} \right)$ linear forests. Since a full resolution is currently out of sight, it is natural to consider proving the linear arboricity conjecture for some interesting classes of graphs, such as random graphs. Building on the method of Alon~\cite{alon1988linear} for the linear arboricity conjecture, McDiarmid and Reed~\cite{arboricity-random-regular} proved it for the class of random regular graphs. In~\cite{GKO:16}, the second author, K\"uhn and Osthus use the aforementioned covering result of Hefetz, K\"uhn, Lapinskas and Osthus~\cite{HKLO:14} to show that whp, the linear arboricity conjecture holds for $G(n,p)$, when $p\ge \frac{\log^{117} n}{n}$. Using Theorem~\ref{thm:main}, we can extend the range of $p$ to the same as in Theorem~\ref{thm:main}. Since the deduction is the same as in~\cite{GKO:16}, we omit the details. 

\begin{cor}
In the same range as above, with high probability $G\sim G(n,p)$ can be covered by $\lceil \frac{\Delta(G)}{2}\rceil$ linear forests.
\end{cor}

%\proof
%Add a new vertex with random edges and apply the theorem. 
%\endproof

\noindent Note that the number of linear forests required in this result is even slightly less than in the linear arboricity conjecture. In general, the $+1$ is needed, for instance if $G$ is regular.

\section{Outline}\label{sec:sketch}

In this section, we briefly sketch our approach. %On a high level, it is similar to the approaches of the previous works.
Since the packing problem is already solved, we might assume that $G\sim G(n,p)$ has a collection of $\lfloor\delta(G)/2\rfloor$ edge-disjoint Hamilton cycles. Let $L$ be the graph consisting of those edges which are not covered by these cycles. Then our goal is now to cover $L$ with $\lceil \Delta(L)/2 \rceil$ Hamilton cycles, where we are allowed to reuse some edges of~$G-L$. Note here that $L$ is much sparser than $G$. Indeed, $G$ has average degree $\Theta(np)$, whereas $\Delta(L)=O(\sqrt{np\log n})$. Yet, it is still too expensive to cover every single edge of $L$ with a new Hamilton cycle (which would be possible using standard tools). Hence, a natural strategy to proceed is to split $L$ into few pieces, and extend each piece to a Hamilton cycle. Since every proper subgraph of a Hamilton cycle is a linear forest, it makes sense to take those as the ``pieces'' we wish to split $L$ into. Clearly, not every linear forest can be extended to a Hamilton cycle, for instance if some end vertices of paths have all their neighbours in the interior of other paths. Hence, what we actually need is a decomposition of $L$ into ``well-behaved'' linear forests, and we will prove that every ``well-behaved'' linear forest can indeed be extended to a Hamilton cycle (see Lemma~\ref{lem:linforestextension}).
Note that with this result, one can already answer the approximate version of the problem, since one can simply split $L$ into $O(\Delta(L))=o(pn)$ matchings. By splitting each such matching randomly into a constant number of smaller matchings, one obtains a collection of well-behaved linear forests, and each of them can be extended into a Hamilton cycle. Hence, together with the packing of $G-L$, this covering uses $(1+o(1))np/2$ Hamilton cycles.

To get the exact bound $\lceil \Delta(G)/2\rceil$, the decomposition of $L$ into linear forests needs a more delicate argument. 
We let $B$ denote the set of vertices which are close to having maximum degree in~$L$. Crucially, this set is rather small, and all vertices have only few neighbours inside~$B$. This allows us to cover all edges incident with $B$ with the optimal number of $\lceil \Delta(L)/2 \rceil$ linear forests (see Lemma~\ref{lem:easylinarb}), which uses a K\H{o}nig--Hall type argument. The maximum degree of $L-B$ is significantly smaller, thus we can decompose it into linear forests using approximate versions of the linear arboricity conjecture. We then show that these two collections of linear forests can be ``merged'' to obtain a single collection of $\lceil \Delta(L)/2 \rceil$ linear forests which decompose~$L$. In fact, recall that we need ``well-behaved'' linear forests. To achieve this, we will actually partition $L-B$ into a large but constant number of subgraphs and apply the approximate version of the linear arboricity conjecture to each subgraph individually.

\section{Preliminaries}
In this section we collect several results in graph theory and random graph theory, whose proofs rely on standard arguments. Before that, we give some definitions and notation.

\subsection{Notation and definitions}
Throughout the paper we use standard graph theoretic notation. Given a graph $G$ we denote by $V(G)$ its vertex set and by $E(G)$ its edge set. Given a subset $S\subseteq V(G)$, we denote by $G[S]$ the subgraph of $G$ induced by $S$ and by $N_G(S)$ the external neighbourhood of $S$ in $G$ (omitting the subscript whenever it is unambiguous). For subsets $S_1,S_2\subseteq V(G)$ we denote by $e_G(S_1,S_2)$ the number of edges with one endpoint in $S_1$ and the other in $S_2$. For a vertex $v\in V(G)$, we denote by $\partial_G(v)$ the set of edges incident to $v$ in $G$. A \emph{cherry} is a path of length 2, whose vertex of degree 2 we call the center of the cherry. A \emph{cherry-matching} is a collection of vertex-disjoint cherries. We denote by $\delta(G)$ and $\Delta(G)$ the minimum and maximum degree of $G$, respectively. A \emph{linear forest} is a graph which consists of a disjoint union of paths. A Hamilton cycle in a graph $G$ is a cycle which contains every vertex in $G$. We let $G(n, p)$ denote the binomial random graph on $n$ vertices, where each edge is included independently with probability $p$.
\subsection{Auxiliary results about linear forests}
%\begin{lem}[Chernoff's bound - see e.g., \cite{alon2016probabilistic}]\label{lem:chernoff}
%Let $X$ be the sum of independent random variables $X_1, \ldots, X_n$ such that $0 \leq X_i \leq k$ for each~$i$. Then, for all $0 < \eps < 1$,
%$$\mathbb{P} \left( X < (1-\eps)\mathbb{E}[X]| \right), \mathbb{P} \left( X > (1+\eps)\mathbb{E}[X]|\right) \leq  e^{-\eps^2 \mathbb{E}[X]/3k^2}.$$
%\end{lem}
We start with a simple lemma which allows us to extend a linear forest by edges of a low maximum degree graph.
\begin{lemma}\label{lem:merginglinfor}
Let $G, H$ be two graphs such that every vertex $v \in V(H)$ has degree at most $d$ in both $H$ and $G$. Then, if there is a collection $\mathcal{F}$ of at least $4d+1$ linear forests $F$ covering the edges of $G$, then there is a collection of $|\mathcal{F}|$ linear forests covering the edges of $G \cup H$.
\end{lemma}
\begin{proof}
We add the edges of $H$ one by one to linear forests in $\mathcal{F}$. Suppose we added $i<e(H)$ edges from $H$ to linear forests in $\mathcal{F}$, and consider an edge $e=(u,v)$ which we did not add yet. Since $\Delta(G\cup H)\leq 2d$, the number of linear forests in which $u$ is not isolated is at most $2d$, and the same holds for $v$. Therefore, there is a linear forest in which both $u$ and $v$ are isolated, so we can add $e$ to that forest. This completes the proof.
\end{proof}

\noindent The following is a simple corollary of Hall's marriage theorem.
\begin{lemma}
\label{thm:hall}
Let $G$ be a bipartite graph with bipartition $V(G)=A\cup B$ such that for all $a \in A, b \in B$ we have $d(b) \geq 2d(a)$. Then, $G$ has a collection of disjoint cherries such that their centers cover $B$.
\end{lemma}
\begin{proof}
Consider the auxiliary bipartite graph $H$ with parts $A,B'$, where $B'$ consists of two disjoint copies of $B$, and $(a,b)$ is an edge if it corresponds to an edge in $G$. Clearly, $d_H(a) =2d_G(a) \leq d_H(b)$ for all $a\in A$ and $b\in B'$. Hence, $|N_H(S)|\geq |S|$ for all $S\subseteq B'$, so Hall's condition is satisfied, thus there is a matching covering $B'$ in $H$. This clearly gives the collection of desired cherries in $G$.
\end{proof}

\noindent The following is the classic theorem of K\H{o}nig.
\begin{theorem}[K\H{o}nig's theorem]\label{thm:konig}
Let $G$ be a bipartite graph with maximum degree $\Delta$. Then, $G$ decomposes into $\Delta$ matchings.
\end{theorem}

\noindent We now combine the above results to prove the following tool, which will allow us to cover the vertices of highest degree in an optimal way with linear forests.

\begin{lemma}\label{lem:easylinarb}
Let $G$ be a graph with a partition $A,B$ of its vertices such that all vertices have degree at most $\Delta(G)/100$ into $B$ and $A$ is an independent set. Then, the edges of $G$ can be decomposed into $\lceil \Delta(G)/2 \rceil$ linear forests.
\end{lemma}
\begin{proof}
Let $\Delta := \Delta(G)$. First, split the graph $G[B]$ into matchings $M_1, \ldots, M_t$ with $t \leq 2\Delta(G[B]) \leq \Delta/50$; this can be done by greedily properly-colouring the edges of $G[B]$. Now, let $B_{\text{high}} \subseteq B$ denote the set of vertices in $B$ with degree at least $\Delta/4$ in $A$. By repeatedly applying Lemma \ref{thm:hall} on the bipartite graph $G[A,B_{\text{high}}]$, since for every $a\in A$ we have $d(a)\leq \Delta/100$, note that we can find $t$ edge-disjoint cherry-matchings $M'_1, \ldots, M'_t$ such that the centers in each cherry matching cover $B_{\text{high}}$. Now, for each $i\in[t]$, we can take a linear forest $F_i$ such that $M_i \subseteq F_i \subseteq M_i \cup M'_i$ and every vertex in $B_{\text{high}}$ has degree two in $F_i$.

Deleting all edges in $\bigcup_{i \leq t} F_i$ gives now a bipartite graph $G'$ with partition $A,B$ whose maximum degree is $\max(\Delta - 2t, \Delta/4) = \Delta - 2t > \Delta/2$. We are only left to decompose the edges of $G'$ into $ \lceil \Delta/2 \rceil - t$ linear forests. For this, let $k :=  \lceil \Delta/2 \rceil - t > \Delta/100$ and define an auxiliary bipartite graph $H = (X,Y)$ as follows: let $X$ consist of two copies $B^{+},B^{-}$ of $B$ and $Y := A$. For each $x \in B$ we have two copies $x^{+} \in B^{+}, x^{-} \in B^{-}$ of it in $X$. For each edge $xy$ in $G'$ we add either the edge $x^{+}y$ to $H$ or the edge $x^{-}y$. Clearly, we can do this in such a way that for all $x \in B$, both $x^{+},x^{-}$ have degree at most $k$ in $H$. Observe then that $H$ has maximum degree at most $k$ since also the vertices of $A$ have degree at most $k$ in $G'$. Therefore, Theorem \ref{thm:konig} implies that its edges can be partitioned into $k$ matchings. To finish, note that the edges of $H$ are in one-to-one correspondence with the edges of $G'$ and that each such matching corresponds to a linear forest in $G'$.
\end{proof}

\noindent We finish this subsection with an approximate version of the linear arboricity conjecture, first shown by Alon~\cite{alon1988linear}.
\begin{theorem}\label{thm:approxLAC}
For every $\eps>0$ there exists $\Delta_0$ such that the following holds for all $\Delta\ge \Delta_0$. Every graph with maximum degree at most $\Delta$ can be decomposed into $(1+\eps)\Delta/2$ linear forests.
\end{theorem}
\subsection{Properties of high degree vertices in random graphs}\label{sec:highdegree}
\noindent In this section we will discuss various properties satisfied by vertices of high degree, that is, close to the maximum degree, in the random graph $G(n,p)$. 
\begin{lemma}[\cite{bollobas1998random}]\label{lem:binomestimate}
Let $X\sim \mathrm{Bin}(n,p)$ with $pn \geq 1$ and $q := 1-p$. Then, if $hqn \geq 3$, we have
$$\mathbb{P} \left(X \geq pn + h \right) < \sqrt{\frac{pqn}{2h^2 \pi}} \cdot \exp\left(-\frac{h^2}{2pqn} + \frac{h}{pqn} + \frac{h^3}{p^2n^2} \right) .$$
\end{lemma}

\begin{lemma}[\cite{bollobas1998random},\cite{krivelevich2012optimal}]\label{lem:maxdegree}\label{lem:degrees in gnp}
Let $n^{-1/2}\ge p\geq 100\log n/n$. With high probability, the minimum degree $\delta$ and maximum degree $\Delta$ of the random graph $G(n,p)$ satisfy
%$$\left|\Delta - pn - \sqrt{2pqn \log n} \right| \leq 2 \log \log n \sqrt{\frac{pqn}{\log n}}$$
\begin{enumerate}[label=\rm{(\roman*)}]
\item $pn+(1-o(1))\sqrt{2pn\log n}\leq \Delta\leq pn+2\sqrt{2pn\log n}$;
\item $pn-2\sqrt{2pn\log n}\leq \delta\leq pn-\frac{1}{2}\sqrt{2pn\log n}$.
\end{enumerate}
\end{lemma}
\begin{proof}
For the first part, the upper bound follows immediately by a union bound over all vertices, since by \Cref{lem:binomestimate} the probability that a vertex has degree at least  $pn+2\sqrt{2pn\log n}$ is at most $o(1/n)$ (by letting $h=2\sqrt{2pn\log n}$).
The lower bound is a direct consequence of Theorem 3.12 in \cite{bollobas1998random} with $m$ being a function which tends to infinity arbitrarily slowly and noting that $\Delta=d_1\geq d_m$. The second part is proven in (\cite{krivelevich2012optimal}, Lemma 2.2).
\end{proof}

\begin{prop}\label{prop:bad vertices}
Let $n^{-2/3} \geq p \geq C \log n/n$ for $C>0$ a large enough constant and let $1/100\geq\alpha>0$. Define $B$ to be the set of vertices with degree at least $pn + (1-\alpha)\sqrt{2pn\log n}$ in $G(n,p)$. Then, with high probability, all of the following hold. 
\begin{enumerate}
    \item $|B| \leq n^{1/10}$.
    \item Every vertex $v$ has at most $100 \leq \sqrt{2pn\log n}/10^{20}$ neighbours in $B\cup N(B-v)$.
\end{enumerate}
\end{prop}
\begin{proof}
%First, note that by Lemma \ref{lem:maxdegree}, whp we have that $\Delta \geq  pn + (1-o(1))\sqrt{2pqn \log n}$. 
Note that if $h := (1-\alpha)\sqrt{2pn\log n} - 100$, then 
$$-\frac{h^2}{2pqn} + \frac{h}{pqn} + \frac{h^3}{p^2n^2} < \left(-1+2\alpha+\frac{10}{\sqrt{C}} \right) \log n ,$$ and so Lemma \ref{lem:binomestimate} implies that
\begin{align}
\label{prob bound}
\mathbb{P} \left(X \geq pn + h\right) < n^{-1+2\alpha+\frac{10}{\sqrt{C}}},
\end{align}
whenever $X$ is a $\text{Bin}(n,p)$ random variable. Since the degree of each vertex in $G(n,p)$ is a $\text{Bin}(n-1,p)$ random variable, it follows that the expected number of vertices with degree at least $pn +h$ is at most $n^{2\alpha+\frac{10}{\sqrt{C}}}$.
Therefore, by Markov's inequality, with high probability there are at most $n^{2\alpha+\frac{11}{\sqrt{C}}} \leq n^{1/10}$ such vertices, so that $|B| \leq n^{1/10}$.

Now we prove that, for any fixed set $S$ of at most $100$ vertices, we have $\prob{S\In B}\le (n^{-1+0.03})^{|S|}$. For this, let us condition on the subgraph $G[S]$. Clearly, this contributes at most $100$ to the degree of each vertex in $S$. With $G[S]$ fixed, for each vertex in $S$ to also be in $B$, it must be that each such vertex has at least $pn+h$ neighbours in $V(G)-S$. Now, by~\eqref{prob bound}, the probability that a given vertex in $S$ has at least $pn+h-100$ neighbours in $V(G)-S$ is at most $n^{-1+0.03}$. Furthermore, since the events of the vertices in $S$ having at least $pn+h-100$ neighbours in $V(G)-S$ are independent, it follows that the probability that $S \subseteq B$ is at most $(n^{-1+0.03})^{|S|}$. 

Note that in particular, letting $A$ denote the set of vertices in $G(n,p)$ with degree at least $pn+h$, this statement implies that $\expn{|A|^{100}} = O(n^{100}) (n^{-1+0.03})^{100}=O( n^3)$. Indeed, $\binom{|A|}{100} = \Theta(|A|^{100})$ counts the number of subsets of size $100$ contained in $A$. By the statement proven in the previous paragraph, each subset of size $100$ is contained in $A$ with probability at most $(n^{-1+0.03})^{100}$. Therefore, $\expn{|A|^{100}} = O\left(\expn{\binom{|A|}{100}}\right) = O(n^{100}) (n^{-1+0.03})^{100} =  O(n^3)$. Applying now Markov's inequality, we have that
$\prob{|A|\ge n^{1/10}} = \prob{|A|^{100}\ge n^{10}}\le \frac{\expn{|A|^{100}}}{n^{10}} \le n^{-6}$. We will need this bound in what follows. 

Now, for the second part we prove that for each vertex $v\in V(G)$, the probability that it has at least $100$ neighbours in  $B\cup N(B-v)$ is at most $o(1/n)$, so that we are then done by a union bound over all $n$ vertices. Fix then a vertex $v$, and expose all edges contained in $V(G)-v$. By the previous paragraph, there are at most $n^{1/10}$ vertices in $G[V-v] \sim G(n-1,p)$ with degree at least $pn+h$, with probability at least $1-n^{-6}$. Call this set of vertices $B'$ and note that $B\subset B'\cup \{v\}$. 
Also note that by standard concentration bounds and a union bound we have that the degree of every vertex in $V-v$ is at most $2np$ with probability at least $1-o(1/n^2)$. This implies that $|B' \cup N(B')| \leq |B'|(1+2np) < n^{0.55}$.
Now we expose the edges touching $v$, and note that the probability that $v$ has at most $100$ neighbours in $B\cup N(B-v)$ is at most $|B'\cup N(B')|^{100}p^{100} \leq n^{55}p^{100} = o(1/n^2)$. Combining all of these events we have that $v$ will have at most $100$ neighbours in $B \cup N(B - v)$ with probability at least $(1-n^{-6})(1-o(1/n^2))(1-o(1/n^2)) = 1 - o(1/n)$, as desired.
\end{proof}

\section{Extending one linear forest}
In this section we show how an appropriate linear forest in a random graph can be extended into a Hamilton cycle. We start with some basic properties typically satisfied by random graphs. 

% \begin{defin}
% A graph $G$ is said to be an \emph{$(s,K)$-expander} (or \emph{$(s,K)$-expanding}) if every subset $S \subseteq V(G)$ of size at most $s$ is such that $|N_G(S)| \geq K|S|$. 
% \end{defin}

\begin{defin}
We say that a graph $G=(V,E)$ has \emph{property $P_\alpha(s,d)$} if for every $X\subseteq V$ of size $|X|\leq s$ and every $F\subseteq E$ such that $|F\cap \partial_{G}(x)|\leq \alpha\cdot d_G(x)$ for every $x\in X$, we have $|N_{G\setminus F}(X)|\geq 2d|X|$.
\end{defin}

The following facts are relatively standard, we include the proofs for completeness.

\begin{lemma}\label{lem:gnp properties}
Let $G\sim G(n,p)$. If $p\geq C\log n/ n$ for large enough $C>0$, then whp $G$ has the following properties.
\begin{enumerate}[label=\rm{(\alph*)}]
\item\label{p:alpha joint} For every two disjoint sets $A,B$ of size at least $\frac{n\log\log n}{\log n}$ we have $e(A,B)>0$.
\item\label{p:edge distribution} For every two disjoint sets $A,B$ with $|A|\geq n/\log n$ and $|B|\geq 100n/C$ we have that $e(A,B)\geq 2|B|$.
\item\label{p:robust expansion1} For $\alpha = 1 - (\log n)^{-1/8}$ and $d=(\log\log n)^{10^4}$, $G$ has the $P_\alpha\big(\frac{n}{\log n\log\log n},d\big)$-property.
\item\label{p:robust expansion2} For $\alpha = 1 - (\log \log \log n)^{-1/8}$ and $d=50\frac{\log\log n}{\log\log\log n}$, $G$ has the $P_\alpha\big(\frac{n}{\log\log n},d\big)$-property.
\end{enumerate}
\end{lemma}

\begin{proof}
    For the first claim, note that for fixed sets $A$ and $B$ of the given size, the expected number of edges is $\mathbb{E}[e(A,B)] = p|A||B|  \geq \frac{Cn(\log\log n)^2}{\log n}$. Hence, by standard Chernoff bounds we have that $$\mathbb P[e(A,B)=0]\leq e^{-\mathbb{E}[e(A,B)]/10}\leq e^{- \frac{3n(\log\log n)^2}{\log n}}.$$

\noindent On the other hand, the number of pairs $A,B$ of the specified size $s= \frac{n\log\log n}{\log n}$ is bounded by $\binom{n}{s}^2\leq (en/s)^{2s} \leq (\log n)^{2s}\leq e^{\frac{2n(\log\log n)^2}{\log n}}$, so by a union bound over all such pairs we complete the proof of the first part. The proof for the second part is analogous. For each such pair $A,B$ we have that the expectation $\mathbb{E}[e(A,B)]$ is at least $100n$. Therefore, 
$$\mathbb P[e(A,B)<2|B|]\leq e^{-\mathbb{E}[e(A,B)]/10}\leq e^{-10n}.$$
Combined with a union bound over at most $2^{2n}$ pairs, we complete this part.

For the third part, we first show that every set $S\subseteq V(G)$ of size at most $\frac{n(2d+1)}{\log n\log\log n}$ spans at most $\eps |S|np/d$ edges, where $\eps := (\log n)^{-1/4}$. Note that the probability that there is such an $S$ which spans at least $\varepsilon|S|np/d$ edges is, by a union bound, at most
\begin{align*}
    &\sum_{s=1}^{\frac{n(2d+1)}{\log n\log \log n}}\binom{n}{s}\binom{s^2/2}{\varepsilon snp/d}p^{\varepsilon snp/d} \leq \sum_{s \geq 1} \left(\frac{e n}{s}\right)^s\left(\frac{e sdp}{2 \varepsilon n p}\right)^{\varepsilon snp/d}=\sum_{s\geq 1} \left(\frac{e n}{s}\left(\frac{e sd}{2\varepsilon n}\right)^{ \varepsilon n p/d}\right)^s\\ &\leq \sum_{s\geq 1} \left(\frac{e n}{s}\left(\frac{e sd}{2 \varepsilon n}\right)^{\varepsilon \log n/d}\right)^s \leq \sum_{s\geq 1} \left(\left(\frac{100s}{n}\right)^{\varepsilon \log n/d-1}\left(\frac{d}{\varepsilon }\right)^{\varepsilon \log n/d}\right)^s \\&\leq\sum_{s \geq 1} \left(\left(\frac{\log n}{d}\right)^{-\varepsilon \log n/d+1}\left(\frac{d}{\varepsilon }\right)^{\varepsilon \log n/d}\right)^s\rightarrow 0,
\end{align*}
\noindent where we have used that $esd < 2 \eps n$ and that since the last geometric sum tends to $0$ because $\left(\frac{\log n}{d}\right)^{-\varepsilon \log n/d+1}\left(\frac{d}{\varepsilon }\right)^{\varepsilon \log n/d} \rightarrow 0$.
Hence, with high probability there is no such set $S$. We can now use this to show that the third property of the lemma holds with high probability. 

Indeed, note first that we have that with high probability, $\delta(G)\geq \frac{np}{2}$ by Lemma \ref{lem:maxdegree}. Suppose for sake of contradiction that $G$ does not have property $P_\alpha(\frac{n}{\log n\log\log n},d)$. Then, by definition, there must exist a set $X$ and a set of edges $F$ such that $|N_{G\setminus F}(X)| < 2d|X|$ and $|F\cap N_{G}(x)|\leq \alpha\cdot d_G(x)$ for every $x\in X$. This implies that the set $S := X \cup N_{G\setminus F}(X)$ which has size $|S| \leq (2d+1)|X| \leq \frac{n(2d+1)}{\log n\log\log n}$ spans at least $(1-\alpha)\delta(G)|X|/2 \geq (1-\alpha)|S|np/10d \geq \frac{|S|np}{d(\log n)^{1/4}}$ edges. In turn, the previous paragraph asserts that the probability that such an $X$ exists tends to $0$. Hence, the third part also holds with high probability.

For the last part, we will again prove that a more general statement holds with high probability. We show that with high probability for every two disjoint sets $S$ and $S'$ such that $|S|\leq \frac{n}{\log\log n}$ and $|S'|\leq 2d|S|$, the number of edges in $G[S\cup S']$ with at least one vertex in $S$ is at most $\eps |S|np$ for $\eps := (\log \log \log n)^{-1/4}$.
Clearly, we can upper bound by a union bound the probability that this event does not hold by at most
\begin{align*}
    &\sum_{s = 1}^{\frac{n}{\log \log n}}\binom{n}{s}\binom{n}{2ds}\binom{2ds^2+s^2}{\varepsilon snp}p^{\varepsilon snp}\leq \sum_{s \geq 1 }\binom{n}{2ds}^2\binom{2ds^2+s^2}{\varepsilon snp}p^{\varepsilon snp} \\
    &\leq \sum_{s \geq 1}\left(\frac{en}{2ds}\right)^{4ds} \left(\frac{3edps^2}{\eps s n p}\right)^{\varepsilon snp}\leq \sum_{s \geq 1} \left(\left(\frac{en}{2ds}\right)^{4d}\left(\frac{3esd}{\varepsilon n}\right)^{\varepsilon n p}\right)^s
       \leq \sum\left(20^d\left(\frac{n}{ds}\right)^{4d-\varepsilon np} \left( \frac{3e}{\eps}\right)^{\eps np}\right)^s
       \\ &\leq \sum\left(\left(\frac{\eps^2 n}{1000ds}\right)^{-\varepsilon np/2}\right)^s\rightarrow 0
\end{align*}
where we used that $d \leq \log \log n \leq \eps n p/10$, that $3esd < \eps n$. Further, the last geometric sum tends to $0$ since $\eps^2 n > 1000ds$. Now, to show the fourth part, suppose for sake of contradiction that the graph does not satisfy the $P_\alpha\big(\frac{n}{\log\log n},d\big)$-property. Then, there must exist a set $X$ of size at most $\frac{n}{\log \log n}$ and a set of edges $F$ such that $|N_{G\setminus F}(X)| < 2d|X|$ and $|F\cap N_{G}(x)|\leq \alpha\cdot d_G(x)$ for every $x\in X$. Denoting now $X$ by $S$ and $N_{G\setminus F}(X)$ by $S'$ we have that there are at least $(1-\alpha)\delta(G)|S|/2 \geq (1- \alpha)np|S|/4 > \eps np |S|$ edges in $G[S \cup S']$ which touch $S$. But since $|S| \leq \frac{n}{\log \log n}$, by the previous analysis this can only occur with probability tending to $0$, which completes the proof.
\end{proof}

\noindent The following result comes from the Friedman--Pippenger tree embedding technique, and is proven in  (\cite{draganic2022rolling}, Theorem 3.5). It states that given a collection of pairs of vertices in a graph with good expansion properties, if the pairs are nicely distributed in the graph, one can connect each pair with mutually vertex disjoint paths.
% There exists an $\varepsilon>0$ such that the following holds. Let $G$ be a $(|G|/10,3)$-expanding $n$-vertex graph with $\delta(G)\geq d_0(\varepsilon)$ such that every two sets $A,B$ of size at least $|G|/20$ have $e_G(A,B) > 0$. Then, given a collection of at most $\varepsilon n/\log n$ vertex-disjoint pairs $(a_i,b_i)$ such that every vertex in a pair has at least $ $ neighbours outside of the pairs, there exists a collection of vertex disjoint $a_ib_i$-paths $P_i$.
\begin{theorem}\label{Thm:VDP}
Let $G$ be a graph with the $P_\alpha(m,d)$ property for some $3\leq d<m$,  such that for every two disjoint $U,V\subseteq V(G)$ of sizes $|U|,|V|\geq m(d-1)/16$ there exists an edge between $U$ and $V$. Let $S$ be any set of vertices such that $|N_G(x)\cap S|\leq \beta d_G(x)$ for every $x\in V(G)$, where $\beta < 2\alpha - 1$, and
let $P=\{\{a_i,b_i\}\}$ be a collection of at most $\frac{dm\log d}{15\log m}$ disjoint pairs from $S$. Then there exist vertex-disjoint paths in $G$ between every pair of vertices $\{a_i,b_i\}$, such that the length of each path is $2\left\lceil\frac{\log (m/16)}{\log (d-1)}\right\rceil+3$. 
\end{theorem}

\noindent A graph is \emph{Hamilton-connected} if for every two vertices in the graph, there is a Hamilton path between them.
We also use the following result of Hefetz, Krivelevich and Szabó (\cite{hefetz2009hamilton}, Theorem~1.2), which states that sufficiently good expander graphs that have an edge between every two large enough sets are Hamilton-connected.
\begin{theorem}\label{lem:hamilton-connected}
Let $G$ be a graph on $n$ vertices for sufficiently large $n$. Suppose that for some $d=d(n)$ the following hold:
\begin{itemize}
    \item For every $S \subseteq V(G)$, if $|S|\leq \frac{n\log\log n\log d}{d \log n\log\log\log n}$ then $|N(S)|\geq d|S|$.
\item There is an edge in $G$ between any two disjoint subsets $A,B \subseteq V(G)$ such that $|A|,|B|\geq \frac{n\log \log n\log d}{4130\log n\log \log \log n}$.
\end{itemize}
Then, $G$ is Hamilton-connected.
\end{theorem}

\noindent The following result is the key tool which allows us to extend a given linear forest $F$ into a Hamilton cycle in a random graph, under the condition that every vertex in the random graph has many neighbours outside of $V(F)$.

% \begin{defin}
% A graph $G$ is said to be \emph{pseudorandom} if
% \begin{enumerate}
%     \item $\Delta(G) \leq 100 \delta(G)$.
%     \item Every set $S \subseteq V(G)$ has $\geq e(S) \geq $.
%     \item Every two sets $A,B \subseteq V(G)$ have $\geq e(A,B) \geq $. 
% \end{enumerate}
% \end{defin}

\begin{lemma}\label{lem:linforestextension}
Let $G\sim G(n,p)$ for $p\geq C\log n/n$ for a large enough $C$. With high probability the following holds. Suppose that $F$ is a linear forest with $V(F)\In V(G)$ such that $|N_G(v)\sm V(F)|\ge np/10^{9}$ for all $v\in V(G)$, and $|V(G)\sm V(F)|\geq n/10^5$.
Then there exists a Hamilton cycle in $G\cup F$ which covers~$F$.
\end{lemma}
\begin{proof}
We assume that $G$ has the properties from \Cref{lem:gnp properties}.
Let $\eps=1/10^{10}$. Let $U=V(G)\setminus V(F)$. Split $U$ uniformly at random into three parts $U_1,U_2,U_3$, where each vertex is assigned a set independently and uniformly at random, and note that whp each set is of size at least $n/10^6$ and for all vertices $v \in V(G)$ and $i \in \{1,2,3\}$ we have that $|N_G(v) \cap U_i| \geq np/10^{10}$. This follows by a standard application of Chernoff bounds. We first do the following process: if $F$ has more than $\eps n$ components, consider a set $S \subseteq V(F)$ of endpoints which contains exactly one endpoint of each path in $F$; since $|S| > \eps n$ and by \Cref{lem:gnp properties} \ref{p:alpha joint}, there is an edge in $S$ which then can be used to update $F$ and reduce its number of components by one; then, repeat the process until $F$ has at most $\eps n$ components. 

Now, we will use $U_1$ and perform another process. At each step, we update our current linear forest $F$ to one which has one less component and one more vertex. We do this until $F$ has at most $n/\log n$ components. Indeed, at each step consider a set $S \subseteq V(F)$ of endpoints which contains exactly one endpoint of each path in $F$, and the set $U_1 \setminus V(F)$; since at each previous step at most one vertex is added to the linear forest, we have that $U_1 \setminus V(F)$ has size at least $|U_1| - \eps n \geq \eps n$ (indeed, recall that we started the process with $F$ having at most $\eps n$ components). Therefore, by \Cref{lem:gnp properties} \ref{p:edge distribution}, we have $e(S,U_1 \setminus V(F)) > |U_1 \setminus V(F)|$ and so, there is a vertex in $U_1 \setminus V(F)$ with at least two neighbours in $S$; observe that we can then add this vertex to the linear forest $F$ while also reducing its number of components by one, as desired. 

At the end of both of these processes, we have a linear forest $F'$ which covers the original linear forest $F$, such that $V(F') \subseteq V(F) \cup U_1$ and it has at most $n/ \log n$ components. We now apply Theorem~\ref{Thm:VDP} in order to transform the linear forest $F'$ into a path. For this, define the graph $G'$ to be the induced subgraph of $G$ on the vertex set $I\cup U_2$ where $I$ is the set of all endpoints of the paths in $F'$. Order the paths of $F'$ as $P_1, \ldots, P_l$ and for each $P_i$ let $a_i,b_i$ denote its endpoints. We want to connect each $b_i$ to $a_{i+1}$ with vertex disjoint paths in $G'$. For this we want to use Theorem \ref{Thm:VDP}. Indeed, we can use the fact that $U_2$ ensures that every vertex has at least $np/10^{10}$ neighbours in $U_2 \subseteq V(G')$ and \Cref{lem:gnp properties} \ref{p:alpha joint} and \ref{p:robust expansion2} to see that $G'$ satisfies the conditions of this theorem with $m=\frac{n}{\log\log n}$, $d=50\frac{\log\log n}{\log\log\log n}$, $\alpha=1-10^{11}(\log \log \log n)^{-1/8}$ and $\beta=1-1/10^{11}$. By part \ref{p:alpha joint}, we have that every two disjoint sets in $G'$ of size at least $m(d-1)/16 \geq \frac{n \log \log n}{\log n}$ have an edge between them. By part \ref{p:robust expansion2} and the fact that $G'$ has minimum degree at least $np/10^{10} \geq \delta(G)/10^{11}$, one can easily check that since $G$ satisfies property $P_{\alpha'}\left(\frac{n}{\log \log n},d \right)$ with $\alpha'$ as in \ref{p:robust expansion2}, then $G'$ must satisfy property $P_{\alpha}\left(\frac{n}{\log \log n},d \right)$ with $1-\alpha = 1-10^{11} \alpha'$. Finally, the number of paths in $F'$ is at most $n/ \log n \leq \frac{dm \log d}{15 \log m}$ and every vertex $x$ in $I$ has $|N_{G'}(x) \cap I| \leq \beta d_{G'}(x)$ since all vertices have at least $\delta(G)/10^{11}$ neighbours in $U_2$. 
Concluding, there exists a collection of $b_ia_{i+1}$-paths in $G'$ which are vertex-disjoint. Note that adding these to $F'$ produces a path $P \subseteq V(G) \setminus U_3$ which covers $F'$. 

To finish, we can apply Theorem~\ref{lem:hamilton-connected} implying that the path $P$ can be extended into a Hamilton cycle. Indeed, let $x,y$ denote the endpoints of the path $P$. Now, note that $G'=G-(V(P) \setminus \{x,y\})$ satisfies the conditions of \Cref{lem:hamilton-connected}. Indeed, using that every vertex has $np/10^{10} \geq \delta(G)/10^{11}$ neighbours in $U_3\subseteq V(G')$ and \Cref{lem:gnp properties}  \ref{p:robust expansion1}, it must be that for $d = (\log \log n)^{10^4}$ every set $S\subset V(G')$ of size at most 
\[\frac{n}{\log n\log\log n}\geq \frac{n\log\log n\log d}{d \log n\log\log\log n}\]
expands by a factor of $d$. Also, by \Cref{lem:gnp properties} \ref{p:alpha joint} we have an edge between every two disjoint sets of size at least $\frac{n\log\log n}{\log n}<\frac{n\log \log n\log d}{4130\log n\log \log \log n}$, so the conditions of \Cref{lem:hamilton-connected} are satisfied. To complete the proof, we connect $x,y$ by a Hamilton path in~$G'$.
\end{proof}

\begin{remark}
Note that, as mentioned in Section~\ref{sec:sketch}, Lemma~\ref{lem:linforestextension} can be used easily to turn an almost-optimal Hamilton packing into an almost-optimal Hamilton cover. Indeed, one can just partition the leftover graph (after packing) into few matchings. Splitting the matchings randomly into two parts, and applying Lemma~\ref{lem:linforestextension} to extend each of the obtained matchings into a Hamilton cycle, would complete the proof.
\end{remark}

\section{Proof of Theorem \ref{thm:main}}
In this section we prove the main theorem. 
% As we outlined, the proof has 3 steps and we will have a subsection for each of these 
% \vspace{0.3cm}
% \noindent \textbf{Applying the initial packing result}
% \vspace{0.3cm}
Let $G\sim G(n,p)$. We will from now on assume that $G$ is fixed and satisfies all the properties appearing in Theorem~\ref{thm:packing}, \Cref{lem:maxdegree} and Proposition~\ref{prop:bad vertices} (with $\alpha:=1/100$) and Lemma~\ref{lem:linforestextension}.
As mentioned in the proof outline, we first apply the packing result (Theorem~\ref{thm:packing}) to obtain $\lfloor \delta(G)/2\rfloor$ edge-disjoint Hamilton cycles. 
Let $L$ denote the graph consisting of the remaining edges. We will use the following parameters in the rest of the section: $t=10^4$, $\alpha=1/450$ and $ k=\lceil\Delta(L)/2\rceil$.
Our goal is to find $k$ Hamilton cycles in $G$ which cover the edges of $L$. 
Let $B$ denote the set of vertices $v$ such that $d_L(v) \geq (1- \alpha)\Delta(L)$. We note the following properties of $L,B$.
\begin{lemma}\label{lem:propafterpacking}
The following hold.
\begin{enumerate}[label=\rm{(\roman*)}]
\item $\sqrt{2pn\log n}\leq \Delta(L)\leq 4\sqrt{2pn\log n}$,
\item $|B| \leq n^{1/10}$,
%, $|N(B)| \leq n^\varepsilon np$ and $B$ is an independent set;
\item $e_G(v,B\cup N_L(B-v)) \le  \Delta(L)/10^{20}$ for all $v\in V(G)$.\label{small degree into bad set}
% \item $|N_G(v)\sm (B\cup N_L(B))|\ge np/100$ for all $v\in V(G)$.\label{outside degrees}
\end{enumerate}
\end{lemma}
\begin{proof}
These follow directly from \Cref{lem:maxdegree} and Proposition~\ref{prop:bad vertices}, noting that $d_L(v)\geq (1-\alpha)\Delta(L)$ implies $d_G(v)\geq \Delta(G)-\alpha\Delta(L)\geq np+(1-\frac{1}{100})\sqrt{2pn\log n}$.
\end{proof}

\subsection{Initial covering of $L$ by linear forests}

\noindent\textbf{Covering the edges of $L-B$ by $(1-\alpha/2)k$ linear forests}

\vspace{0.3cm}

\noindent 
%Now, firstly note the by definition of $B$, the graph $L-B$ has maximum degree at most $(1-\alpha)\Delta(L)$ and therefore, Theorem \ref{thm:approxLAC} implies that there is a collection $\mathcal{F}_1$ of $(1-\alpha + o(1))\Delta(L)$ linear forests covering the graph $L-B$.
The following is a simple lemma which enables us to split the edges of $L-B$ into a constant number of graphs $L_i$, such that each of them has roughly the same maximum degree, and such that each $v\in V(L_i)$ has many neighbours in $V(G)\sm V(L_i)$.
\begin{lemma}\label{lem:linear forests with reservoirs}
There exists a partition of the edges of $L-B$ into subgraphs $L_1,\dots,L_t$ with
\begin{align}
\Delta(L_i)\le (1-3\alpha/4)\frac{\Delta(L)}{t-2} \label{partition max deg}
\end{align}
for each $i\in [t]$; further, denoting $R_i=V(G)\sm (B\cup V(L_i))$, we have for every vertex $v\in V(G)$ that $e_G(v,R_i\sm N_L(B-v))\geq np/200t$ and $|R_i|\geq n/2t$. 
% \begin{align}
% |N_G(v)\sm (B\cup N_L(B)\cup V(L_i)|\ge np/200t \label{partition degrees}
% \end{align}
% for each $i\in [t]$ and all $v\in V(G)$.
\end{lemma}

\begin{proof}
% Partition $V(G)$ randomly into sets $U_1,\dots,U_t$. For each $i$, partition $G[U_i]$ randomly into $t-1$ parts and label each of them by a unique element from $[t]\sm\Set{i}$. For each pair $1\le i<j\le t$, partition $G[U_i,U_j]$ randomly into $t-2$ parts and label each of them by a uniqe element in $[t]\sm\Set{i,j}$. For each $i$, let $L_i$ be the subgraph consisting of all the parts which are labeled with $i$, with the vertices of $B$ removed.
Partition $V(L)$ randomly into sets $R_1',\dots,R_t'$, where each vertex is assigned a part independently and uniformly at random. Define $R_i$ to be the set $R_i'\setminus B$.
Label each edge in each $L[R_i]$ uniformly at random with one of the colors in $[t]\setminus\{i\}$. For each pair $1\le i<j\le t$, label each edge in $L[R_i,R_j]$ uniformly at random with an element in $[t]\sm\Set{i,j}$. Let $L_i$ be the graph on $V(L)\setminus (B\cup R_i)$ which contains all the edges labeled by $i$.

Now, the required bounds simply follow from standard concentration inequalities. Indeed, note first that for each $i$ and vertex $v \in V(L)$, we have that the edges in $L$ incident to $v$ are independently colored with color $i$, each with probability at most $\frac{1}{t-2}$, hence by using Chernoff bounds we have that $d_{L_i}(v)\leq (1-3\alpha/4)\Delta(L)/(t-2)$ with probability at least $1-\text{exp}(-\Delta(L)/10^{9}t)$. Since $\Delta(L)\geq \sqrt{2pn \log n} \geq \sqrt{C}\log n$ and since $t=10^4 < \sqrt{C}/10^{10}$, this probability is at least $1-o(1/tn)$ and thus, a union bound over all $i$ and $v$ implies that with high probability, $\Delta(L_i)\le (1-3\alpha/4)\frac{\Delta(L)}{t-2}$ for each~$i$.

Since $R_i'$ is a random subset of $V(L)$ where each vertex is included with probability $1/t$, the expected size of $|R_i'|$ is $n/t$. Hence,
by a standard Chernoff bound we have that $|R_i'|\geq 2n/3t$ whp. Hence, since $|B|\leq n^{1/10}$, we have $|R_i|=|R_i'\setminus B|\geq n/2t$ whp.

 Finally, we want to show that whp $e_G(v,R_i\sm N_L(B-v))\geq np/200t$ for all $i\in[t]$. First recall that $\delta(G)\geq np/2$ by Lemma~\ref{lem:degrees in gnp}. We then have that $\mathbb{E}[e_G(v,R_i')]\geq np/2t\geq 10\log n$, hence whp $e_G(v,R_i')\geq np/4t$ for all vertices $v$ and $i$, again using a Chernoff bound, and a union bound over at most $n$ vertices, and over all $i\in[t]$. Now, by Lemma~\ref{lem:propafterpacking}~\ref{small degree into bad set} we have that $e_G(v, B\cup N_L(B-v))$ is at most $\Delta(L)/10^{20}<np/10t$. Since
 $R_i=R_i'\setminus B$, we have that with high probability, $e_G(v,R_i\sm N_L(B-v))\geq np/200t$ for all $v$ and $i$, as desired.
\end{proof}

%There exists a partition of $V(G)$ into sets $U_1,\dots,U_t$ such that $d_G(v,U_i)\approx np/t$ for all $v\in V(G)$ and $i\in[t]$.

%Goal for extended version: There exists a partition of $L-B$ into subgraphs $L_1,\dots,L_t$ and a collection $\cF_0$ of $k:=\lceil \frac{\Delta}{2}\rceil$ linear forests which partition $L_0:=L[B\cup N_L(B)]$ with the following properties:
%\begin{align}
%\Delta(L_i)\le (1-\alpha)\frac{\Delta}{t-1} \mbox{ for each }i\in [t] \\
%\mbox{for any $i\in [t]$ and $F\in \cF_0$, we have that $|N_G(v)\sm (V(F)\cup V(L_i)|\ge np/2t$}
%\end{align}

\noindent Now, for each $i\in[t]$, Theorem~\ref{thm:approxLAC} applied to $L_i$ gives us a collection of at most $(1+\eps) \frac{\Delta(L_i)}{2} $ linear forests which cover~$L_i$, for any choice of a constant $\varepsilon>0$.

\begin{defin}\label{def:F_0}
     For each $i\in[t]$ (recall $t=10^4$), define a collection $\cF_i$ of $(1-2\alpha/3)\frac{\Delta(L)}{2(t-2)}$ linear forests, as given by Theorem~\ref{thm:approxLAC}, which cover~$L_i$. We denote by $\mathcal{F}_0$ the union of these collections. Note that $ |\mathcal{F}_0|=t\cdot (1-2\alpha/3)\frac{\Delta(L)}{2(t-2)}\leq (1-\alpha/2)k$, so we can assume that $\mathcal F_0$ has precisely $(1-\alpha/2)k$ linear forests by possibly adding forests to $\mathcal{F}_0$ which consist of any single vertex $v\notin B$.
\end{defin}

\noindent\textbf{Covering the rest of $L$ by $k$ linear forests}
\vspace{0.3cm}

\noindent We will now cover the rest of the edges in $L$, that is, those edges incident with $B$. 
\begin{lemma}
There exists a collection $\mathcal{F}_1$ of $k$ linear forests which decompose the edges of the subgraph of $L$ consisting of the edges touching $B$.  
\end{lemma}
\begin{proof}
Given the properties in Lemma \ref{lem:propafterpacking}, we need only to apply Lemma \ref{lem:easylinarb} to this subgraph of $L$, which we denote as $L_B$. Indeed, take $A:=V(L)\sm B$ and $B:= B$ and note that $A$ is independent in $L_B$, while the degree into $B$ is at most $\Delta(L)/1000$ by Lemma~\ref{lem:propafterpacking}~\ref{small degree into bad set}. This gives a collection $\mathcal{F}_1$ of $k$ linear forests which decompose the edges of $L_B$ as desired. 
\end{proof}
\subsection{Joining $\mathcal{F}_0$ to $\mathcal{F}_1$}

\noindent We will now merge $\mathcal{F}_0$ with $\mathcal{F}_1$ in order to create one collection $\mathcal{F}$ of $k$ linear forests covering the edges of $L$, such that each vertex in each forest has a large neighbourhood outside of that forest in~$G$. In the end, we will transform each of these linear forests into a Hamilton cycle using Lemma \ref{lem:linforestextension}, thus finishing the proof. The construction of $\mathcal{F}$ proceeds as follows.

First, recall that by Definition~\ref{def:F_0}, we have $|\mathcal{F}_0| = (1-\alpha/2)k$. Take a subcollection $\mathcal{F}'_1 \subseteq \mathcal{F}_1$ of size $|\mathcal{F}_0|$ and take an arbitrary pairing $(F_0,F_1)$ between $\mathcal{F}_0$ and $\mathcal{F}'_1$. For each pair $(F_0,F_1)$, we can define the linear forest $F_2 := (F_0 \cup F_1) \setminus E$, where $E$ is the set of edges in $F_0$ which are incident to vertices also incident to some edge of $F_1$. We now add all these new linear forests $F_2$ to $\mathcal{F}$, so that currently, $|\mathcal{F}| = |\mathcal{F}_0| = (1-\alpha/2)k$. Let $L' \subseteq L$ denote the subgraph of $L$ formed by the edges in $L-B$ which do not belong to any linear forest $F_2 \in \mathcal{F}$, that is, those which were in some set $E$ as defined before. Then, the following holds.
\begin{claim}
$\Delta(L') \leq \Delta(L)/10^{19}$.
\end{claim}
\claimproof
Let us fix a vertex $x \notin B$. Suppose an edge $xy$ belongs to $L'$, that is, there is a pair $(F_0,F_1)$ such that $xy \in E$. Then, it must be that one of $x,y$ belongs to $V(F_1) \setminus B$, and thus, to $N_L(B)$. The number of such vertices $y$ with $y \in V(F_1) \setminus B$ is at most $e_L(x,N_L(B))$. If $x \in V(F_1) \setminus B$, then there are at most $2e_L(x,B)$ possible options for $y$ - this is because choosing a neighbour of $x$ in $B$ will fix the linear forest $F'_1$ which then gives two options for $y$ with $xy \in F_0$. Hence $\Delta(L')\leq \max_{x} 3e_L(x,B\cup N_L(B))\leq \Delta(L)/10^{19}$, using property \ref{small degree into bad set} of Lemma~\ref{lem:propafterpacking}.
\qed
\\\\
\noindent Let $L'_i := L' \cap L_i$ for each $1 \leq i \leq t$. We will now add linear forests to $\mathcal{F}$ which cover the edges in $L'$ as well as the edges used in $\mathcal{F}''_1 := \mathcal{F}_1 \setminus \mathcal{F}'_1$. We do this so that at the end we have $|\mathcal{F}| \leq k$ as desired. Let us partition $\mathcal{F}''_1$ arbitrarily into $t$ collections $\mathcal{F}''_1(1), \ldots, \mathcal{F}''_1(t)$ each of size at least $\alpha k/2t >\Delta(L)/10^8\geq  100 \Delta(L')$. For each $1 \leq i \leq t$, we can then apply Lemma \ref{lem:merginglinfor} with $H:= L'_i$, $G := \bigcup_{F \in \mathcal{F}''_1(i)} F$. Indeed, note that a vertex $v \in V(H)$ does not belong to $B$; therefore, its degree in $H$ is at most $\Delta(L') \leq |\mathcal{F}''_1(i)|/100$ and its degree in $G$ is at most $e_L(v,B) \leq \Delta(L)/10^{20} < |\mathcal{F}''_1(i)|/10^{12}$ using property \ref{small degree into bad set} of Lemma~\ref{lem:propafterpacking}. So, Lemma \ref{lem:merginglinfor} finds us a collection $\mathcal{F}(i)$ of $|\mathcal{F}''_1(i)|$ linear forests covering the edges of $L'_i$ and those used in $\mathcal{F}''_1(i)$. To conclude, we add all the collections $\mathcal{F}(i)$ to the collection $\mathcal{F}$, so that $|\mathcal{F}| \leq k$ and the linear forests in it cover all edges of $L$. 

\begin{claim}\label{lem:forests are good}
    Every $F \in \mathcal{F}$ satisfies the conditions of Lemma \ref{lem:linforestextension}, i.e., $|V(G) \setminus V(F)|\geq n/10^5$ and $|N_G(v) \setminus V(F)| \ge np/10^9$ for all $v \in V(G)$.
\end{claim}

\claimproof   
First, recall that each $F \in \mathcal{F}$ is such that $F[V(G) \setminus B] \subseteq L_i$ for some $i$, so that $R_i\sm N_L(B-v) \subseteq V(G) \setminus V(F)$. Further, by Lemma \ref{lem:linear forests with reservoirs} we have that $|R_i|\geq n/2t$, by Proposition~\ref{prop:bad vertices} we know that $|B|\leq n^{1/10}$, and by  Lemma~\ref{lem:maxdegree} we have $\Delta(G)\leq 2np$.
Combining these three, and recalling that $p\leq n^{-2/3}$, we get that $$|V(G) \setminus V(F)| \geq |R_i\sm N_L(B-v)| \geq n/2t-2np|B|\geq n/2t-o(n)\geq n/10^5.$$ For the second part, we can similarly apply Lemma \ref{lem:linear forests with reservoirs} since it tells us that for all vertices $v$ we have $|N_G(v) \setminus V(F)| \geq e_G(v,R_i\sm N_L(B-v))\geq np/200t \geq np/10^7$, as desired.
\qed

\subsection{Extending $\mathcal{F}$ into a Hamilton cycle cover}

\noindent To finish the proof, we can take the collection $\mathcal{F}$ of at most $k$ linear forests covering the edges of $L$. Claim~\ref{lem:forests are good} shows that each linear forest $F\in \mathcal{F}$ satisfies the conditions of Lemma \ref{lem:linforestextension}. So we can apply this lemma in order to extend each of these linear forests to a Hamilton cycle in $G$. This together with the covering of $G-L$ as discussed in the beginning of the proof gives a covering of $G$ with $\lceil \Delta(G)/2 \rceil$ Hamilton cycles as desired.

\section{Concluding remarks}
In this paper we showed that a random graph $G\sim G(n,p)$ can, with high probability, be covered by $\lceil\Delta(G)/2\rceil$ Hamilton cycles for all $p \geq C\log n/n$ for some large constant $C$, thus proving a conjecture of Glebov, Krivelevich and Szab\'o \cite{GKS:14} in a strong form and improving upon the previous results of Hefetz, K\"uhn, Lapinskas and Osthus~\cite{HKLO:14} and of Ferber, Kronenberg and Long \cite{ferber2017packing}. An interesting direction of research would then be to consider different spanning structures, such as spanning trees or $H$-factors for a fixed graph $H$. Some work on this topic has already been done by Bal, Frieze, Krivelevich and Loh \cite{bal2013packing}.

Another interesting direction concerns the directed versions of the packing and covering problems. For example, does our result extend to random digraphs? It was shown by Ferber, Kronenberg and Long \cite{ferber2017packing} that an asymptotically optimal covering of a random digraph $D(n,p)$ with directed Hamilton cycles exists whenever $p\gg \frac{\log^2n}{n}$. It would be interesting to see if the range of $p$ can be improved, as well as if an optimal covering can be found. Many of the ideas discussed here are also applicable to the directed context, and it would be worthwhile to explore the adaptation of our methods in this case.

Finally, it is well-known that $G(n,p)$ is already Hamiltonian when $p\geq (1+\varepsilon)\log n/n$, hence it is natural to ask whether the covering result also holds for this range of $p$. It would also be interesting to consider a hitting time version, since we know that in the random graph process, the graph is Hamiltonian as soon as it has minimum degree~$2$. This problem was first formulated in~\cite{HKLO:14}.

\bibliographystyle{plain}

\end{document}